\documentclass[12pt,a4paper,english,reqno]{amsart}
\usepackage{amsthm}
\usepackage{amssymb}
\usepackage{amscd}
\usepackage{amstext}
\usepackage[all]{xy}
\xyoption{2cell}
\UseTwocells

\usepackage{graphicx}
\numberwithin{equation}{section}
\numberwithin{figure}{section}

\theoremstyle{plain}
\newtheorem{thm}{Theorem}[section]

\newtheorem*{thm-nonum}{Theorem}
\newtheorem{lem}[thm]{Lemma}
\newtheorem{prp}[thm]{Proposition}
  
  \theoremstyle{definition}
  \newtheorem{dfn}[thm]{Definition}
    \newtheorem{exm}[thm]{Example}

  \newtheorem{rmk}[thm]{Remark}
  \theoremstyle{plain}
  
  \theoremstyle{plain}

\newcommand{\xyR}[1]{%
\xydef@\xymatrixrowsep@{#1}}

\newcommand{\xyC}[1]{%
\xydef@\xymatrixcolsep@{#1}}

\def\al{\alpha}
\def\be{\beta}
\def\ga{\gamma}
\def\de{\delta}
\def\ep{\varepsilon}

\def\et{\eta}

\def\la{\lambda}

\def\si{\sigma}

\def\De{\Delta}

\def\Ph{\Phi}

\def\Ker{\operatorname{Ker}}

\def\calC{{\mathcal C}}
\def\calD{{\mathcal D}}

\def\calM{{\mathcal M}}

\def\bbZ{{\mathbb Z}}
\def\bbP{{\mathbb P}}

\def\incl{\hookrightarrow}
\def\iso{\cong}

\def\ovl{\overline}
\def\Ds{\bigoplus}

\def\dsm#1,#2..#3{\bigoplus_{{#1}={#2}}^{#3}}
\def\sm#1,#2..#3{\sum_{{#1}={#2}}^{#3}}

\def\id{1\kern-.25em{\text{{\rm l}}}} 

\def\isoto{\ \raise.8ex\hbox{$^{\sim}$}\kern-.7em\hbox{$\to$}\ } 

\def\ang#1{{\langle #1 \rangle}}

\def\bg{%
\family{cmr}\size{20}{12pt}\selectfont}

\def\bigzerou{%
\smash{\lower1.7ex\hbox{\bg 0}}}

\def\repr[#1;#2;#3;#4;#5]{
\left(
\begin{matrix}#1\\#2\end{matrix}
#3
\begin{matrix}#4\\#5\end{matrix}
\right)}
\numberwithin{equation}{section}
\numberwithin{figure}{section}

\def\kCat{\Bbbk\text{-}\mathbf{Cat}}

\def\Gr{\operatorname{Gr}}
\def\k{\Bbbk}

\def\incl{\hookrightarrow}

\newtheorem{clm}{Claim}
\def\gr{\operatorname{Gr}}
\def\dis{\displaystyle}
\def\rar{\rightarrow}
\def\kCat{\Bbbk\text{-}\mathbf{Cat}}
\def\bp{\overline{\Ph}}
\def\ol{\overline}
\def\dia{\De (A)}
\def\tX{\tilde{X}}
\def\Rel{\operatorname{Rel}}

\begin{document}

\title{Presentations of Grothendieck constructions}
\author{Hideto Asashiba and Mayumi Kimura}

\address{
Department of Mathematics,
Faculty of Science,
Shizuoka University,
836 Ohya, Suruga-ku,
Shizuoka, 422-8529, Japan}


\thanks{This work is partially supported by Grant-in-Aid for Scientific Research
(C) 21540036 from JSPS}

\maketitle

\begin{abstract}
We will give quiver presentations of the Grothendieck constructions
of functors from a small category
to the 2-category of $\Bbbk$-categories for a commutative ring $\Bbbk$.
\end{abstract}

\section*{Introduction}
Throughout this paper $I$ is a small category, $\k$ is a commutative ring,
and $\kCat$ denotes the the 2-category of all $\k$-categories, $\k$-functors between them
and natural transformations between $\k$-functors.

The Grothendieck construction is a way to form a single category $\Gr(X)$ from
a diagram $X$ of small categories
indexed by a small category $I$, 
which first appeared in \cite[\S 8 of  Expos{\'e} VI]{SGA71}.
As is exposed by Tamaki \cite{Tam09}
this construction has been used as a useful tool in homotopy theory (e.g., \cite{Tho79}) or
topological combinatorics (e.g., \cite{WZZ99}).
This can be also regarded as a generalization of orbit category construction from
a category with a group action.

In \cite{Asa-a} we defined a notion of derived equivalences of (oplax) functors from $I$ to $\kCat$,
and in \cite{Asa-b} we have shown that if (oplax) functors $X, X' \colon I \to \kCat$
are derived equivalent,
then so are their Grothendieck constructions
$\Gr(X)$ and $\Gr(X')$.
An easy example of a derived equivalent pair of functors is given by using diagonal functors:
For a category $\calC$ define the {\em diagonal} functor $\De(\calC) \colon I \to \kCat$
to be a functor sending all objects of $I$ to $\calC$
and all morphisms in $I$ to the identity functor of $\calC$.
Then if categories $\calC$ and $\calC'$ are derived equivalent,
then so are their diagonal functors $\De (\calC )$ and $\De (\calC')$.
Therefore, to compute examples of derived equivalent pairs using this result,
it will be useful to present Grothendieck constructions of functors by quivers with relations.
We already have computations in two special cases.
First for a $\k$-algebra $A$, which we regard as a $\k$-category with a single object,
we noted in \cite{Asa-b} that
if $I$ is a semigroup $G$, a poset $S$,
or the free category $\bbP Q$ of a quiver $Q$,
then the Grothendieck construction $\gr (\De (A))$
of the diagonal functor $\De(A)$ is isomorphic to
the semigroup algebra $AG$,
the incidence algebra $AS$,
or the path-algebra $AQ$, respectively.
Second in \cite{Asa11} we gave a quiver presentation of the 
orbit category $\calC/G$ for each $\k$-category $\calC$ with an action
of a semigroup $G$ in the case that $\k$ is a field,
which can be seen as a computation of a quiver presentation 
of the Grothendieck construction $\Gr(X)$ of each functor $X\colon G \to \kCat$.

In this paper we generalize these two results as follows:
\begin{enumerate}
\item We compute the Grothendieck construction $\Gr(\De(A))$
of the diagonal functor $\De(A)$ for each $\k$-algebra $A$
and each small category $I$.
\item We give a quiver presentation of the Grothendieck construction $\Gr(X)$
for each functor $X \colon I \to \kCat$ and each small category $I$
when $\k$ is a field.
\end{enumerate}

In section 1 we give necessary definitions and recall the fact that all categories can
be presented by quivers and relations.
Sections 2 and 3 are devoted to the computation (1) and a quiver presentation (2)
above, respectively.
Finally in section 4 we give some examples.

\section{Preliminaries}

Throughout this paper
$Q=(Q_0,Q_1,t,h)$ is a quiver, where 
$t(\al) \in Q_0$ is the {\em tail} and $h(\al) \in Q_0$ is the {\em head} of each arrow $\al$ of $Q$.
For each path $\mu$ of $Q$,
the tail and the head of $\mu$ is denoted by $t(\mu)$ and $h(\mu)$, respectively.
For each non-negative integer $n$ 
the set of all paths of $Q$ of length at least $n$ is denoted by
$Q_{\ge n}$.
In particular $Q_{\ge 0}$ denotes the set of all paths of $Q$.

A category $\calC$ is called a {\em $\k$-category} if for each $x,y\in \calC$,
$\calC (x,y)$ is a $\k $-module and the compositions are $\k$-bilinear.

\begin{dfn}
Let $Q$ be a quiver.
\begin{enumerate}
\item
The {\em free} category $\bbP Q$ of $Q$ is the category whose underlying quiver is $(Q_0, Q_{\ge 0},t, h)$
with the usual composition of paths.
\item
The {\em path} $\k$-category of $Q$ is the $\k$-linearization
of $\bbP Q$ and is denoted by $\k Q$.
\end{enumerate}
\end{dfn}

\begin{dfn}
Let $\calC$ be a category.
We set
$$
\Rel(\calC):=\bigcup _{(i,j)\in \calC_0\times \calC_0}\calC(i,j)\times \calC(i,j),
$$
elements of which are called {\em relations} of $\calC$.
Let $R \subseteq \Rel(\calC)$.
For each $i,j\in \calC _0$ we set
$$
R(i,j):= R \cap (\calC(i,j)\times \calC(i,j)).
$$
\begin{enumerate}
\item
The smallest congruence relation
$$R^c:=\bigcup _{(i,j)\in \calC _0\times \calC _0}
\{(d a c,d b c)\mid c\in \calC (-,i),d\in \calC(j,-),
(a,b)\in R(i,j)\}$$
containing $R$ is called the {\em congruence relation} generated by $R$. 
\item
For each $i, j \in \calC_0$ we set
\begin{align*}
R^{-1}(i,j)&:=\{(g,f)\in \calC (i,j)\times \calC (i,j) \mid (f,g)\in R(i,j)\}\\
1_{\calC (i,j)}&:=\{(f,f)\mid f\in \calC (i,j)\}\\
S(i,j)&:=R(i,j)\cup R^{-1}(i,j)\cup 1_{\calC (i,j)}\\
S(i,j)^1&:=S(i,j)\\
S(i,j)^n&:=\{(h,f)\mid \exists g \in \calC(i,j), (g,f)\in S(i,j),(h,g)\in S(i,j)^{n-1}\}
\quad(\text{for all }n\ge 2)\\
S(i,j)^{\infty} &:=\dis\bigcup _{n\geq 1}S(i,j)^n
\text{, and set}\\
R^e&:=\bigcup _{(i,j)\in \calC _0\times \calC _0}S(i,j)^{\infty}.\end{align*}
$R^e$ is called the {\em equivalence relation} generated by $R$.\item
We set $R^\#:=(R^c)^e$ (cf.\ \cite{Ho95}).
\end{enumerate}
\end{dfn}

\begin{rmk}
In the statement (2) above,
$S(i,j)^{\infty }$ is the smallest equivalence relation on $\calC (i,j)$ containing $R(i,j)$ for each $i,j\in \calC _0$. 
\end{rmk}

\begin{dfn}\label{teigi2}
Let $\calC$ be a category and $R \subseteq \Rel(\calC)$. 
Then a category\ $\calC/R^\#$ is defined as follows:
\begin{enumerate}
\item[(i)]$(\calC/R^\#)_0:=\calC _0$. 
\item[(ii)]For $i,j\in (\calC/R^\#)_0$, $(\calC/R^\#)(i,j):=\calC(i,j)\slash R^\#(i,j)$. \\
For each $f\in (\calC/R^\#)(i,j)$, we set $\ol{f}$ the equivalence class of $f$ in $R^\#$. 
\item[(iii)]For $i,j,k\in (\calC/R^\#)_0$ and $\ol{f}\in (\calC/R^\#)(i,j)$, $\ol{g}\in (\calC/R^\#)(j,k)$, $\ol{g}\circ \ol{f}:=\ol{g\circ f}$. 
\item[(iv)]A functor\ $F:\calC \rar \calC/R^\#$ is defined as follows:
\begin{enumerate}
\item[(a)]For $i\in \calC _0$, $F(i)=i$. 
\item[(b)]For $i,j\in \calC (i,j)$ and $f\in \calC (i,j)$, $F(f)=\ol{f}$. 
\end{enumerate}
\end{enumerate}
\end{dfn}

\begin{rmk}
In definition\ref{teigi2}, $R^\#$ is a congruence relation,
therefore the composition in (iii) is well-defined. 
\end{rmk}

The following is well known (cf.\ \cite{ML}).

\begin{prp}
Let $\calC$ be a category, and $R \subseteq \Rel(\calC)$. 
Then the category\ $\calC/R^\#$ and the functor\ $F:\calC \rar \calC/R^\#$
defined above satisfy the following conditions. 
\begin{enumerate}
\item[(i)]For each $i,j\in \calC _0$ and each $(f,f')\in R(i,j)$ we have $Ff=Ff'$. 
\item[(ii)] If a functor\ $G:\calC \rar \calD$ satisfies $Gf=Gf'$ for all $f,f'\in \calC(i,j)$
and all $i,j\in \calC_0$ with $(f,f')\in R(i,j)$,
then there exists a unique functor\ $G':\calC/R^\#\rar \calD$ such that $G'\circ F=G$. 
\end{enumerate}
\end{prp}

\begin{dfn}
Let $Q$ be a quiver and $R \subseteq \Rel(\bbP Q)$.
We set 
$$
\ang{Q\mid R}:= \bbP Q/R^\#.
$$
\end{dfn}

The following is straightforward.
\begin{prp}
Let $\calC$ be a category, $Q$ the underlying quiver of $\calC$, and
set
$$R:=\{(e_i,\id_i),(\mu ,[\mu ])\mid i\in Q_0,\mu \in Q_{\geq 2}\}\subseteq \Rel(\bbP Q),$$
where $e_i$ is the path of length $0$ at each vertex $i \in Q_0$,
and $[\mu]:=\al_n\circ\dots\circ\al_1$ $($the composite in $\calC)$ for
all paths $\mu =\al_n\dots\al_1 \in Q_{\ge 2}$ with
$\al_1,\dots,\al_n\in Q_1$.
Then
$$\calC\iso \ang{Q\mid R}.$$
\end{prp}


By this statement, an arbitrary category is presented by a quiver and relations. 
Throughout the rest of this paper $I$ is a small category
with a presentation $I=\langle Q\mid R\rangle$.

\section{Grothendieck constructions of Diagonal functors}

\begin{dfn}
Let  $X:I\rar \kCat$ be a functor. Then a category $\gr (X)$, called the {\em Grothendieck construction} of $X$, is defined as follows:
\begin{enumerate}
\item[(i)] $(\gr (X))_0:=\dis \bigcup _{i\in I_0}\{(i,x)\mid x\in X(i)_0\}$
\item[(ii)] For $(i,x),(j,y)\in (\gr (X))_0$
$$\gr (X)((i,x),(j,y)):=\dis \bigoplus _{a\in I(i,j)}X(j)(X(a)x,y)$$
\item[(iii)] For $f=(f_a)_{a\in I(i,j)}\in \gr (X)((i,x),(j,y))$ and $g=(g_b)_{b\in I(j,k)}\in \gr (X)((j,y),(k,z))$
$$g\circ f:=\left(\dis \sum_{\begin{subarray}{c}c=ba\\ a\in I(i,j)\\ b\in I(j,k)\end{subarray}}g_bX(b)f_a\right)_{c\in I(i,k)}$$
\end{enumerate}
\end{dfn}

\begin{dfn}
Let $\calC \in \kCat_0$.
Then the {\em diagonal functor} $\De(\calC)$ of $\calC$ is a functor from $I$ to $\kCat$ sending each arrow $a\colon i\to j$ in $I$  to
$\id_{\calC}\colon \calC \to \calC$
in $\kCat$. 
\end{dfn}
In this section, we fix a $\k$-algebra $A$ which we regard
as a $\k$-category with a single object $*$ and
with $A(*,*)=A$.
The {\em quiver algebra} $AQ$ of $Q$ over $A$ is the $A$-linearization
of $\bbP Q$, namely $AQ:=A\otimes_\k \k Q$.

\begin{dfn}
The ideal of $AQ$ generated by the elements
$g-h$ with $(g,h)\in R$ is denoted by $\ang{R}_A$:
$$\langle R\rangle_A :=AQ\{g-h\mid (g,h)\in R\}AQ.$$
\end{dfn}

The purpose of this section is to prove the following theorem which computes the Grothendieck construction $\Gr(\De(A))$
of $\dia :I\to \kCat$.

\begin{thm}\label{Gr-diag}
$\gr (\dia )\iso AQ\slash \langle R\rangle_A$.
\end{thm}

To prove this theorem, we use the following two lemmas.

\begin{lem}\label{ho}
Let $S$ be a set,
$E\subseteq S\times S$ an equivalence relation on $S$.
Then
$$(\bigoplus _{x\in S}Ax)\slash (\sum _{(g,h)\in E }A(g-h))\iso \bigoplus _{\ol{x}\in S\slash E}A\ol{x}$$
\end{lem}
\begin{proof}
Let
$\ep :\dis \bigoplus _{x\in S}Ax
\rar
\bigoplus _{\ol{x}\in S\slash E}A\ol{x}$ be
a homomorphism of $A$-modules defined
by $x\mapsto \ol{x}$ ($x \in S$).
Then the sequence
$$0\rar \sum _{(g,h)\in E }A(g-h) \incl \bigoplus _{x\in S}Ax\xrightarrow{\ep } \bigoplus _{\ol{x}\in S\slash E}A\ol{x} \rar 0$$
is exact. 
Indeed, since $\ep$ is obviously a surjection by definition,
it is enough to show that
$\Ker \ep =\dis \sum _{(g,h)\in E }A(g-h)$. \\
For each $(g,h)\in E$ we have
$$\ep (g-h)=\ol{g-h}=\ol{g}-\ol{h}=0,$$
and hence $\dis \sum _{(g,h)\in E}A(g-h)\subseteq \Ker \ep$. \\
To prove the reverse inclusion,
let $\dis \sum _{x\in S}a_xx\in \Ker \ep $ ($a_x \in A)$.
Then
since 
$$0= \ep \left(\sum _{x\in S}a_xx\right)=\sum _{x\in S}a_x\ol{x}=\sum _{\ol{x}\in S\slash E}\sum _{x'\in \ol{x}}a_{x'}\ol{x},$$
we have $\sum _{x'\in \ol{x}}a_{x'}=0$
for each $\ol{x}\in S\slash E$,
and hence for each $x \in S$ we have
$$a_{x} = -\sum _{x'\in \ol{x}\backslash\{x\}}a_{x'}$$
and
$$\sum _{x'\in \ol{x}}a_{x'}x'=a_xx+\sum _{x'\in \ol{x}\setminus \{x\}}a_{x'}x'=\sum _{x'\in \ol{x}\setminus \{x\}}a_{x'}(x'-x).$$
Let $L$ be a complete set of representatives in $S\slash E$.
Then we have
$$\sum _{x\in S}a_xx=\sum _{x\in L}\sum _{(x,x')\in E\backslash \{(x,x)\}}a_{x'}(x'-x).$$
Hence $\Ker \ep \subseteq \dis \sum _{(g,h)\in E}A(g-h)$
and we have $\Ker \ep =\dis \sum _{(g,h)\in E}A(g-h)$. 
\end{proof}

We will gives an explicit form of $\langle R\rangle_A$
as follows.

\begin{lem}\label{da}
For each $i,j\in Q_0$, 
$$\langle R\rangle_A(i,j)=\sum_{(g,h)\in R^\# (i,j)}A(g-h)$$
\end{lem}
\begin{proof}
We set $I(i,j):=\dis \sum_{(g,h)\in R^\# (i,j)}A(g-h)$. First, we prove that $I(i,j)$ is an ideal of $AQ$. 
It is obvious that $I(i,j)$ is closed under addition.
Let $a \in AQ(i',i)$, $b \in AQ(j,j')$, $c\in I(i,j)$.
Then there exist $a_\al,b_\be,c_{g,h}\in A$ such that
\begin{eqnarray*}
a&=&\sum_{\al \in \bbP Q(i'i)}a_\al \al \\
b&=&\sum_{\be \in \bbP Q(j,j')}b_\be \be \\
c&=&\sum_{(g,h)\in R^\#(i,j)}c_{g,h}(g-h)
\end{eqnarray*}
and
\begin{eqnarray*}
bca&=&\left(\sum_{\be \in \bbP Q(j,j')}b_\be \be\right)\left(\sum_{(g,h)\in R^\#(i,j)}c_{g,h}(g-h)\right)\left(\sum_{\al \in \bbP Q(i'i)}a_\al \al\right)\\
     &=&\sum_{\de \in \bbP Q(i',j')}\sum_{\begin{subarray}{c}\de =\be \ga \\ \ga \in \bbP Q(i',j)\\ \be \in \bbP Q(j,j')\end{subarray}}\sum_{\begin{subarray}{c}\ga =(g-h)\al \\ \al \in \bbP Q(i',i)\\(g,h)\in R^\# (i,j)\end{subarray}}b_\be c_{g,h}a_\al \de.
\end{eqnarray*}
By $(g,h)\in R^\#$, $(\be g\al ,\be h\al )\in R^\#$. Hence
$bca \in AQ(i',j')$ as desired.

Next, we prove that $\ang{R}_A(i,j)=I(i,j)$. 
Since $R\subseteq R^\#$, for each $(g,h)\in R(i,j)$ we have
$$g-h\in I(i,j).$$
Hence $\langle R\rangle_A(i,j)\subseteq I(i,j)$.
Further for each $(g,h)\in R^c(i,j)$, 
there exist $(g',h')\in R(i',j')$, $e\in \bbP Q(i,i')$ 
and $f\in \bbP Q(j',j)$ such that
$$(g,h)=(fg'e,fh'e).$$
Then
$$g-h=fg'e-fh'e=f(g'-h')e\in \langle R\rangle_A(i,j).$$
Hence also for each $(g,h)\in R^\#(i,j)$ we have $g-h\in I(i,j)$ 
because $I(i,j)$ is closed under addition.
Therefore $I(i,j)\subseteq \langle R\rangle_A(i,j)$, 
and hence $\langle R\rangle_A(i,j)=I(i,j)$. 
\end{proof}

\subsection*{Proof of Theorem\ref{Gr-diag}}
The object classes and the morphism spaces of
$\gr(\dia)$ and $AQ/\ang{R}_A$ are given as follows.

$\gr (\dia )$:
\begin{enumerate}
\item[(i)]$\gr (\dia )_0=\{ (i,*)\mid i\in Q_0\} $. 
\item[(ii)]For $(i,*),(j,*)\in \gr (\dia )_0$
\begin{eqnarray*}
\gr (\dia )((i,*),(j,*))&=&\bigoplus _{a\in I(i,j)}\dia(j)(\dia(a)(*),*)\\
                       &=&\bigoplus _{a\in I(i,j)}A(*,*)=A^{(I(i,j))}
\end{eqnarray*}
\end{enumerate}
$AQ\slash \langle R\rangle_A$:
\begin{enumerate}
\item[(i)]$(AQ\slash \langle R\rangle_A)_0=Q_0$. 
\item[(ii)]For $i,j\in (AQ\slash \langle R\rangle_A)_0$
\begin{eqnarray*}
(AQ\slash \langle R\rangle_A) (i,j)&=&(\bigoplus _{a\in \bbP Q(i,j)}Aa)\slash \langle R\rangle_A (i,j)\\
&=&(\bigoplus _{a\in \bbP Q(i,j)}Aa)\slash \sum_{(g,h)\in R^\#(i,j)}A(g-h)\\
&=&\Ds_{a\in I(i,j)}Aa
\end{eqnarray*}
\end{enumerate}
by Lemma \ref{da} and the last equality
is given by the isomorphism in Lemma \ref{ho}.
We define a functor $F:\gr (\dia )\rar AQ\slash \langle R\rangle_A$ by
$$(i,*)\mapsto i$$
$$(f_a)_{a \in I(i,j)}\mapsto \sum_{a \in I(i,j)}f_a a$$
for each $(f_a)_{a \in I(i,j)}: (i,*) \to (j,*)$ in $\gr (\dia )$. 
We check that $F$ is well-defined as a $\k$-linear functor.
For each $(i,*)\in \gr (\dia )_0$ we have
\begin{eqnarray*}
F(\id _{(i,*)})&=&F((\de _{1_ia})_{a\in I(i,i)})\\
                 &=&\sum_{a\in I(i,i)}\de _{1_ia}a\\
                 &=&1_i
\end{eqnarray*}
For each $f\in \gr (\dia )((i,*,),(j,*))$ and $g\in \gr (\dia )((j,*,),(k,*))$, 
there exist $f_a,g_b\in A$ ($a\in I(i,j),b\in I(j,k)$) such that
\begin{eqnarray*}
f&=&(f_a)_{a\in I(i,j)}\\
g&=&(g_b)_{b\in I(j,k)}.
\end{eqnarray*}
Then
\begin{eqnarray*}
F(g\circ f)&=&F\left(\left(\sum_{\begin{subarray}{c}c=ba\\ a\in I(i,j)\\ b\in I(j,k)\end{subarray}}g_bf_a\right)_{c\in I(i,k)}\right)\\
              &=&\sum_{c\in I(i,k)}\left(\sum_{\begin{subarray}{c}c=ba\\ a\in I(i,j)\\ b\in I(j,k)\end{subarray}}g_bf_a\right)c
\end{eqnarray*}
\begin{eqnarray*}
F(g)F(f)&=&\left(\sum_{b\in I(j,k)}g_bb\right)\left(\sum_{a\in I(i,j)}f_aa\right)\\
          &=&\sum_{c\in I(i,k)}\left(\sum_{\begin{subarray}{c}c=ba\\ a\in I(i,j)\\ b\in I(j,k)\end{subarray}}g_bf_a\right)c\\
          &=&F(g\circ f).
\end{eqnarray*}
Hence $F$ is a functor.
Obviously $F$ is $\k$-linear. 
It is clear that $F$ is bijective on objects and that 
for each $i,j\in Q_0$, $F$ induces an isomorphism
$$\gr (\dia )((i,*),(j,*))\to (AQ\slash \langle R\rangle_A)(i,j)$$
by the definition of $F$.
Therefore $\gr (\dia )\iso AQ\slash \langle R\rangle_A$. 
\qed

\begin{rmk}
Theorem \ref{Gr-diag} can be written in the form
$$\gr (\dia )\iso A\otimes _\Bbbk (\Bbbk Q\slash \langle R\rangle _\Bbbk).$$
\end{rmk}

\section{The quiver presentation of Grothendieck constructions}

In this section we give a quiver presentation of the Grothendieck construction of an arbitrary functor
$I \to \kCat$.  Throughout this section we assume that $\k$ is a field.

\begin{thm}\label{qv-pres-Gr}
Let $X:I\rar \kCat$ be a functor, and for each $i\in I$ set $X(i)=\Bbbk Q^{(i)}\slash \langle R^{(i)}\rangle$ with $\Ph^{(i)}\colon \k Q^{(i)} \to X(i)$
the canonical morphism, where $R^{(i)}\subseteq \Bbbk Q^{(i)}$, $\langle R^{(i)}\rangle \cap \{e_x\mid x\in Q{(i)}_0\}=\emptyset$. 
Then Grothendieck construction is presented by the quiver with relations
$(Q, R')$ defined as follows. 

Quiver: $Q'=(Q'_0,Q'_1,t',h')$, where
\begin{enumerate}
\item[(i)] $Q'_0:=\dis \bigcup _{i\in I}\{{}_ix\mid x\in Q_0^{(i)}\} $. 
\item[(ii)] $Q'_1:=\dis \bigcup _{i\in I}\{ \{{}_i\al \mid \al \in Q_1^{(i)}\}\cup \{ (a,{}_ix):{}_ix\rar {}_j(ax)\mid a:i\rar j\in Q_1,x\in Q_0^{(i)},ax\neq 0\} \} $,\\
where we set $ax:=X(\ol{a})(x)$. 
\item[(iii)] For $\al \in Q_1^{(i)}$, $t'({}_i\al )=t^{(i)}(\al )$ and $h'({}_i\al )=h^{(i)}(\al )$. 
\item[(iv)] For $a:i\rar j\in Q_1,x\in Q_0^{(i)}$, $t'(a,{}_ix)={}_ix$ and $h'(a,{}_ix)={}_j(ax)$. 
\end{enumerate}

Relations: $R':=R'_1\cup R'_2\cup R'_3$, where 
\begin{enumerate}
\item[(i)] $R'_1:=\{\si ^{(i)}(\mu )\mid i\in Q_0,\mu \in R^{(i)}\}$, \\
where we set $\si^{(i)}:\Bbbk Q^{(i)}\hookrightarrow \Bbbk Q'$. 
\item[(ii)] $R'_2:=\{\pi (g,{}_ix)-\pi (h,{}_ix)\mid i,j\in Q_0,(g,h)\in R(i,j),x\in Q_0^{(i)}\}$,
where for each path $a$ in $Q$ we set
$$\pi (a,{}_ix):=(a_n,{}_{i_{n-1}}(a_{n-1}a_{n-2}\dots a_1x))\dots (a_2,{}_{i_1}(a_1x))(a_1,{}_ix)$$ 
if $a=a_n\dots a_2a_1$ for some $a_1, \dots, a_n$ arrows in $Q$, and
$$
\pi (a,{}_ix):= e_{{}_ix}
$$
if $a=e_i$ for some $i \in Q_0$.
\item[(iii)] $R'_3:=\{(a,{}_iy){}_i\al -{}_j(a\al )(a,{}_ix)\mid a:i\rar j\in Q_1,\al :x\rar y\in Q_1^{(i)}\}$,
where we take $a\al :ax\rar ay$ so that $\Ph ^{(j)}(a\al) \in X(\ol{a})\Ph ^{(i)}(\al)$:
\[\xymatrix{\ \al \in \Bbbk Q^{(i)}  \ar[r]^(0.60){\Ph ^{(i)}} & X(i)\ \ \  \ar @<-2.5mm>[d]^{X(\ol{a})}\\
                 a\al \in \Bbbk Q^{(j)}  \ar[r]^(0.60){\Ph ^{(j)}} & X(j).\ \ \ }\]
\end{enumerate}
Note that the ideal $\langle R'\rangle$ is 
independent of the choice of $a\al$
because $R'_1\subseteq R'$. 
\end{thm}
\begin{proof}
We define a $\k$-functor $\Ph :\Bbbk Q'\rar \gr (X)$ by:
\begin{enumerate}
\item[(i)]for ${}_ix\in Q'_0$, $\Ph ({}_ix)=(i,x)$;
\item[(ii)]for ${}_i\al :{}_ix\rar {}_iy\in Q^{(i)}_1$, $\Ph ({}_i\al )=(\de _{\id_ia}\Ph ^{(i)}(\al))_{a\in I(i,i)}$;
\item[(iii)]for $(a,{}_ix):{}_ix\rar {}_j(ax)\in Q'_1$, $\Ph ((a,{}_ix))=(\de _{\ol{a}b}\id_{X(\ol{a})(x)})_{b\in I(i,j)}$;
\item[(iv)]for $\al _n\al_{n-1}\dots \al _1\in \bbP Q'$\ $(\al _1,\dots ,\al _n\in Q')$
$$\Ph (\al _n\al _{n-1}\dots \al _1):=\Ph (\al _n)\Ph (\al _{n-1})\dots \Ph (\al _1);\text{ and}$$
\item[(v)] for $f:=\sum _{\al \in \bbP Q'({}_ix,{}_jy)}f_\al \al \in \Bbbk Q'({}_ix,{}_jy)$
($f_{\al} \in \k$)
$$
\Ph(f):= \sum _{\al \in \bbP Q'({}_ix,{}_jy)}f_\al \Ph(\al).
$$
\end{enumerate}

\begin{clm}
$\Ph $ is well-defined as a $\k$-functor, and  is bijective on objects. 
\end{clm}

Indeed, this is clear by noting that for each ${}_ix\in Q'_0$ we have
\begin{eqnarray*}
\Ph (e_{{}_ix})&=&(\de _{\id_i,a}\Ph ^{(i)}(e_x))_{a\in I(i,i)}\\
                 &=&\id_{(i,x)}.
\end{eqnarray*}

\begin{clm}
$\Ph (R')=0$. 
\end{clm}

Indeed, for each $i\in Q_0$, $\al, \be \in Q^{(i)}_1$ we have
\begin{eqnarray*}
\Ph ({}_i\be {}_i\al )&=&\Ph ({}_i\be )\Ph ({}_i\al )\\
                          &=&(\de_{1_i,b}\Ph ^{(i)}(\be ))_{b\in I(i,i)}
                          (\de_{1_i,a}\Ph ^{(i)}(\al ))_{a\in I(i,i)}\\
                          &=&\left(\sum _{\begin{subarray}{c} c=ba\\a\in I(i,i)\\b\in I(i,i)\end{subarray}}\de_{1_i,b}\Ph ^{(i)}(\be )X(b)(\de_{1_i,a}\Ph ^{(i)}(\al ))\right)_{c\in I(i,i)}\\
                          &=&(\de_{1_i,c}\Ph^{(i)}(\be \al ))_{c\in I(i,i)},
\end{eqnarray*}
which shows that $\Ph(\si^{(i)}(\mu))=
(\de_{1_i,c}\Ph^{(i)}(\mu))_{c\in I(i,i)}$
for each 
$\mu\in \bbP Q^{(i)}$,
and that for each $\mu \in R^{(i)}$,
$$
\Ph (\si^{(i)}(\mu))=(\de_{1_i,a}\Ph ^{(i)}(\mu ))_{a\in I(i,i)}
=(\de_{1_i,a}0)_{a\in I(i,i)}=0.
$$
Thus $\Ph(R'_1)=0$.

For each $g_1:i\rar j,g_2:j\rar k\in Q_1$, ${}_ix\in Q'$,
\begin{eqnarray*}
\Ph (\pi (g_2g_1,{}_ix))&=&\Ph ((g_2,{}_j(g_1x)))\Ph ((g_1,{}_ix))\\
                      &=&(\de_{\ol{g_2},b}\id_{X(\ol{g_2})(g_1x)})_{b\in I(j,k)}(\de_{\ol{g_1},a}\id_{X(\ol{g_1})(x)})_{a\in I(i,j)}\\
                      &=&\left(\sum _{\begin{subarray}{c}c=ba\\ a\in I(i,j)\\ b\in I(j,k)\end{subarray}}\de_{\ol{g_2},b}\id _{X(\ol{g_2})(g_1x)}X(b)(\de_{\ol{g_1},a}\id _{X(\ol{g_1})(x)})\right)_{c\in I(i,k)}\\
                      &=&(\de_{\ol{g_2g_1},c}\id _{X(\ol{g_2})(g_1x)}\id _{X(\ol{g_1})(x)})_{c\in I(i,k)}\\
                      &=&(\de_{\ol{g_2g_1},c}\id _{X(\ol{g_2g_1})(x)})_{c\in I(i,k)},
\end{eqnarray*}
which shows that $\Ph (\pi (g,{}_ix))=
(\de_{\ol{g},b}\id _{X(\ol{g})(x)})_{b\in I(i,j)}$
for each $g\in \bbP Q$.
Therefore 
\begin{eqnarray*}
\Ph (\pi (g,{}_ix)-\pi (h,{}_ix))&=&\Ph (\pi (g,{}_ix))-\Ph (\pi (h,{}_ix))\\
                                       &=&(\de_{\ol{g},b}\id _{X(\ol{g})(x)})_{b\in I(i,j)}-(\de_{\ol{h},a}\id _{X(\ol{h})(x)})_{a\in I(i,j)}\\
                                       &=&0
\end{eqnarray*}
because $\ol{g}=\ol{h}$ for each $(g,h)\in R(i,j)$. Thus $\Ph(R'_2)=0$.

For $a:i\rar j\in Q_1$, $\al :x\rar y\in Q_1^{(i)}$
\begin{eqnarray*}
\Ph ((a,{}_iy){}_i\al )&=&\Ph ((a,{}_iy))\Ph ({}_i\al )\\
                         &=&(\de_{\ol{a},c}\id _{X(\ol{a})(y)})_{c\in I(i,j)}(\de _{1_i,b}\Ph ^{(i)}(\al ))_{b\in I(i,i)}\\
                         &=&\left(\sum _{\begin{subarray}{c}d=cb\\b\in I(i,i)\\c\in I(i,j)\end{subarray}}\de_{\ol{a},c}\id _{X(\ol{a})(y)}X(c)(\de_{1_i,b}\Ph ^{(i)}(\al ))\right)_{d\in I(i,j)}\\
                         &=&(\de_{\ol{a},d}\id _{X(\ol{a})(y)}X(\ol{a})(\Ph ^{(i)}(\al )))_{d\in I(i,j)}\\
                         &=&(\de_{\ol{a},d}X(\ol{a})(\Ph ^{(i)}(\al )))_{d\in I(i,j)},
\end{eqnarray*}
\begin{eqnarray*}
\Ph ({}_j(a\al )(a,{}_ix))&=&\Ph ({}_j(a\al ))\Ph ((a,{}_ix))\\
                             &=&(\de_{1_j,c}\Ph ^{(j)}(a\al ))_{c\in I(j,j)}(\de _{\ol{a},b}\id _{X(\ol{a})(x)})_{b\in I(i,j)}\\
                             &=&\left(\sum _{\begin{subarray}{c}d=cb\\ b\in I(i,j)\\ c\in I(j,j)\end{subarray}}\de_{1_j,c}\Ph ^{(j)}(a\al )X(c)(\de_{\ol{a},b}\id _{X(\ol{a})(x)})\right)_{d\in I(i,j)}\\
                             &=&(\de_{\ol{a},d}\Ph^{(j)}(a\al )X(1_j)(\id _{X(\ol{a})(x)}))_{d\in I(i,j)}\\
                             &=&(\de_{\ol{a},d}\Ph^{(j)}(a\al ))_{d\in I(i,j)}.
\end{eqnarray*}
Since $X(\ol{a})(\Ph^{(i)}(\al))=\Ph^{(j)}(a\al)$ by the choice of $a\al$, 
we have
$$\Ph((a,{}_iy){}_i\al)=\Ph({}_j(a\al)(a,{}_ix)).$$
Hence $\Ph(R'_3)=0$, and finally $\Ph(R')=0$. 

By the claim above we see that $\Ph$ induces a functor\ $\bp:\Bbbk Q'\slash \langle R'\rangle \rar \gr (X)$. We prove that $\bp$ is an isomorphism.
To this end, we first consider a basis of $(\Bbbk Q'\slash \langle R'\rangle )({}_ix,{}_jy)$ for each ${}_ix$, ${}_jy\in Q'_0$.

\begin{clm}
For each $(g,h)\in R^\#(i,j)$ and $x\in Q^{(i)}$, $\ol{\pi(g,{}_ix)}=\ol{\pi(h,{}_ix)}$. 
\end{clm}

Indeed, there exist some $(a,b)\in R(i',j')$, $c\in \bbP Q(i,i')$ and $d\in \bbP Q(j',j)$ such that
$$(g,h)=(dac,dbc).$$
Then
\begin{eqnarray*}
\pi (g,{}_ix)-\pi (h,{}_ix)&=&\pi (dac,{}_ix)-\pi (dbc,{}_ix)\\
                              &=&\pi (d,{}_{j'}(acx))\pi (a,{}_{i'}(cx))\pi (c,{}_ix)-\pi (d,{}_{j'}(bcx))\pi (b,{}_{i'}(cx))\pi (c,{}_ix)\\
                              &=&\pi (d,{}_{j'}(acx))(\pi (a,{}_{i'}(cx))-\pi (b,{}_{i'}(cx)))\pi (c,{}_ix).
\end{eqnarray*}
Therefore since $\pi (a,{}_{i'}(cx))-\pi (b,{}_{i'}(cx))\in R'$, 
we have $\pi (g,{}_ix)-\pi (h,{}_ix)\in R'$. Hence $\ol{\pi(g,{}_ix)}=\ol{\pi(h,{}_ix)}$. 

\medskip

For each $a:i\to j$ in $I$ we define a functor $\tX(a)\colon 
\k Q^{(i)}\to \k Q^{(j)}$ as follows:
\begin{itemize}
\item For each $x\in Q^{(i)}_0$, $\tX(a)(x):=X(\ol{a})(x)$.
\item For each arrow $\al:x\to y$ in $Q^{(i)}$, $\tX(a)(\al):=a\al$.
\item For each path $\mu:=\al_n\dots \al_1$ $(n\ge 2)$ in $Q^{(i)}$, 
$\tX(a)(\mu):=\tX(a)(\al_n)\dots \tX(a)(\al_1)$. 
\end{itemize}

\begin{clm}
For each ${}_ix,{}_jy\in Q'_0$ and $\mu \in \bbP Q'({}_ix,{}_jy)$, 
there exist some $a\in I(i,j)$ and $\nu\in \k Q^{(j)}({}_j(ax),{}_jy)$
such that  $\ol{\mu}=\ol{\nu \pi (a,{}_ix)}$. 
\end{clm}

Indeed, since $(b,{}_kv){}_k\al -{}_l(b\al )(b,{}_ku)\in R'$ for each
$b:k\rar l$ in $Q_1$ and $\al :u\rar v$ in $Q_1^{(k)}$, we have 
$$
\ol{(b,{}_kv){}_k\al }=\ol{{}_l(b\al )(b,{}_ku)},
$$
which implies 
$$
\ol{(b,{}_kv)\si^{(k)}(\la)}=\ol{\si^{(l)}\tX(b)(\la)(b,{}_ku)}
$$
for each $\la\in \k Q^{(k)}({}_ku,{}_kv)$.
By using this formula in the path $\mu$ we can move factors of the form
$\ol{(b,{}_kv)}$ to the right, and finally we have
$$
\ol{\mu}=\ol{\nu(a_t,x_t)\cdots (a_1,x_1)}
$$
for some $0\le t\in \bbZ$, $\nu\in \k Q^{(j)}$, $x_1,\cdots ,x_t\in Q'_0$, $a_1,\cdots, a_t\in Q_1$, where $(a_t,x_t)\cdots (a_1,x_1)$ is a path of length t in $Q'$, and hence we have
$(a_t,x_t)\cdots (a_1,x_1)=\pi(a,x_1)$ ($a:=a_t\cdots a_1$).
Hence we have $\nu\in \k Q^{(j)}({}_j(ax),{}_jy)$ and 
$\ol{\mu}=\ol{\nu \pi (a,{}_ix)}$.

\begin{clm}
$\calM:=\{ \ol{\al \pi (a,{}_ix)}|a\in I(i,j),\al \in \calM _j(ax,y)\} $\  is a basis of $(\Bbbk Q'\slash \langle R'\rangle )({}_ix,{}_jy)$, where $\calM _j(ax,y)$ is a basis of $(\Bbbk Q^{(j)}\slash \langle R^{(j)}\rangle )(ax,y)$. 
\end{clm}

Indeed, assume $\dis\sum _{\begin{subarray}{c}a\in I(i,j)\\ \al \in \bbP Q^{(j)}(ax,y)\end{subarray}}k_{a,\al}\ol{\al \pi (a,{}_ix)}=0$.
Then
\begin{eqnarray*}
\bp \left(\sum _{\begin{subarray}{c}a\in I(i,j)\\ \al \in \bbP Q^{(j)}(ax,y)\end{subarray}}k_{a,\al}\ol{\al \pi (a,{}_ix)}\right)&=&\sum _{\begin{subarray}{c}a\in I(i,j)\\ \al \in \bbP Q^{(j)}(ax,y)\end{subarray}}k_{a,\al}\bp (\ol{\al })\bp (\ol{\pi (a,{}_ix)})\\
                     &=&\sum _{\begin{subarray}{c}a\in I(i,j)\\ \al \in \bbP Q^{(j)}(ax,y)\end{subarray}}k_{a,\al}(\de_{1_j,c}\Ph ^{(j)}(\al ))_{c\in I(j,j)}(\de_{a,b}\id _{X(a)(x)})_{b\in I(i,j)}\\
                     &=&\sum _{\begin{subarray}{c}a\in I(i,j)\\ \al \in \bbP Q^{(j)}(ax,y)\end{subarray}}k_{a,\al}(\sum _{\begin{subarray}{c}d=cb\\ b\in I(i,j)\\ c\in I(j,j)\end{subarray}}\de_{1_j,c}\Ph ^{(j)}(\al )X(c)(\de_{a,b}\id _{X(a)(x)}))_{d\in I(i,j)}\\
                     &=&\sum _{\begin{subarray}{c}a\in I(i,j)\\ \al \in \bbP Q^{(j)}(ax,y)\end{subarray}}k_{a,\al}(\de_{a,d}\Ph ^{(j)}(\al )X(1_j)(\id _{X(a)(x)}))_{d\in I(i,j)}\\
                     &=&\sum _{\begin{subarray}{c}a\in I(i,j)\\ \al \in \bbP Q^{(j)}(ax,y)\end{subarray}}k_{a,\al}(\de_{a,d}\Ph ^{(j)}(\al ))_{d\in I(i,j)}\\
                     &=&\left(\Ph ^{(j)}\left(\sum _{\al \in \bbP Q^{(j)}(ax,y)}k_{d,\al}\al \right)\right)_{d\in I(i,j)}\\
                     &=&0
\end{eqnarray*}
Since $\al \in \calM _j(ax,y)$, we have $k_{d,\al}=0$.
Therefore $\calM$ is a basis of
$(\Bbbk Q'\slash \langle R'\rangle )({}_ix,{}_jy)$. 

Here we define $\si _a \colon X(j)(X(a)(x),y)\hookrightarrow \dis\bigoplus _{a\in I(i,j)}X(j)(X(a)(x),y)$ by $\mu \mapsto (\de_{b,a}\mu )_{b\in I(i,j)}$
for each $\mu \in X(j)(X(a)(x),y)$.
Then a basis of $\gr (X)((i,x),(j,y))$ is written by $\dis\bigcup _{a\in I(i,j)}\si_a(\Ph^{(j)}(\calM _j(ax,y)))$, and for each $\ol{\al \pi (a,{}_ix)}\in \calM$
we have
\begin{eqnarray*}
\bp (\ol{\al \pi (a,{}_ix)})&=&(\de_{a,d}\Ph ^{(j)}(\al ))_{d\in I(i,j)}\\
                                  &=&\si_a\Ph ^{(j)}(\al ).
\end{eqnarray*}
Hence $\bp$ induces an isomorphism $(\Bbbk Q'\slash \langle R'\rangle )({}_ix,{}_jy)\isoto \gr (X)((i,x),(j,y))$. 

Therefore $\bp$ is an isomorphism.
\end{proof}

\begin{rmk}
The description of the proof of Claim 5 in the proof of Theorem 8.1 in \cite{Asa11}
is not complete.  This corresponds to Claim 4 above, and the formula (8.4) in \cite{Asa11} should
be replaced by a linear combination
$$
\ovl{\et} = \sum t_{y, \al_s, \dots}\ovl{e_y\al_s\dots \al_1 (g_t, x_t)\dots (g_1, x_1)}
$$
with $t_{y, \al_s, \dots} \in \k$.
Correspondingly, we must remove ``$\ovl{\et} = $'' in the last formula in Claim 5 there.
The earlier version arXiv:0807.4706v6 of the paper records the correct proof.
\end{rmk}

\section{Examples}

In this section, we illustrate Theorems \ref{Gr-diag} and \ref{qv-pres-Gr} by some examples.

\begin{exm}
Let $Q$ be the quiver
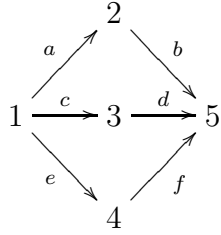
\[\xymatrix{   &2& \\
                   1&3&5\\
                    &4& 
                   \ar ^a"2,1";"1,2"
                   \ar ^b"1,2";"2,3"
                   \ar ^c"2,1";"2,2"
                   \ar ^d"2,2";"2,3"
                   \ar _e"2,1";"3,2"
                   \ar _f"3,2";"2,3" }\]
and let $R=\{(ba,dc)\}$.
Then the category $I:=\ang{Q \mid R}$ is not given as a semigroup,
as a poset or as the free category of a quiver.
For any algebra $A$ consider the diagonal functor $\De(A) \colon I \to \kCat$.
Then by Theorem \ref{Gr-diag} the category $\Gr(\De(A))$ is given by
$$
AQ/\ang{ba-dc}.
$$

\begin{rmk}
Let $Q$ and $Q'$ be quivers having neither double arrows nor loops,
and let $f \colon Q_0 \to Q'_0$ be a map (a {\em vertex map} between $Q$ and $Q'$).
If $Q(x, y) \ne \emptyset$ ($x, y \in Q_0$) implies
$Q'(f(x), f(y)) \ne \emptyset$ or $f(x) = f(y)$, then
$f$ induces a $\k$-functor $\hat{f} \colon \k P \to \k P'$ defined by the following correspondence:
For each $x \in Q_0$, $\hat{f}(e_x):= e_{f(x)}$, and for each arrow $a \colon x \to y$ in $Q$,
$f(a)$ is the unique arrow $f(x) \to f(y)$ (resp.\ $e_{f(x)}$) if $f(x) \ne f(y)$ (resp.\ if $f(x) = f(y)$).
\end{rmk}

\begin{exm}
Let  $I=\langle Q\mid R\rangle$ be as in the previous example.
Define a functor $X:I\to \kCat$ by the $\k$-linearizations of
the following quivers in frames and
the $\k$-functors induced by the vertex maps expressed by broken arrows
between them:
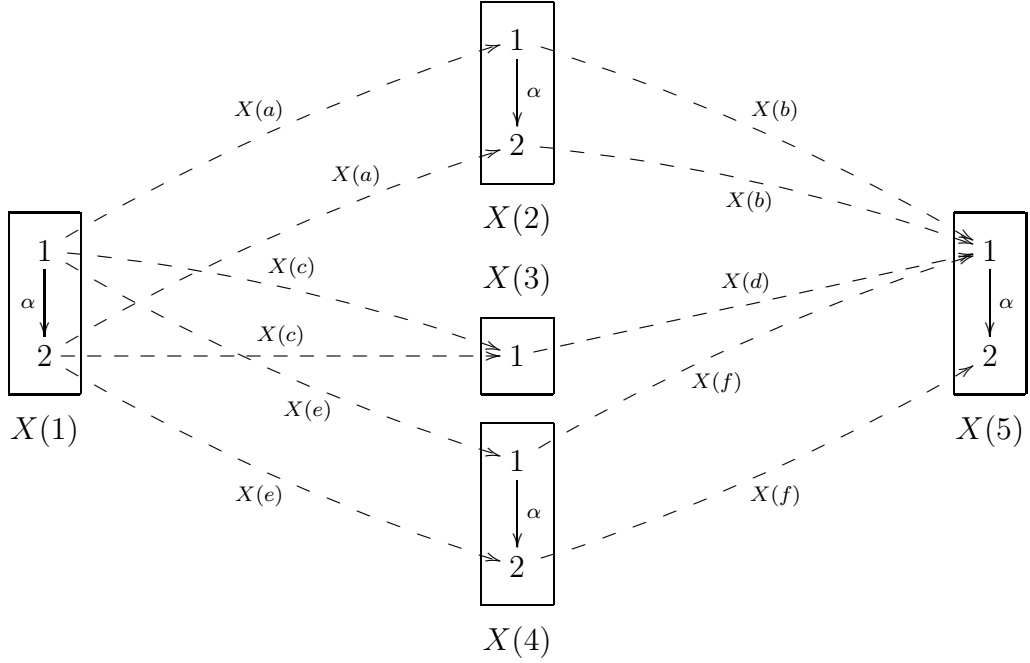
\def\g{\save [].[d]!C*++[F-]\frm{} \restore}
\def\h{\save [].[]!C*++[F]\frm{} \restore}
\[
\xymatrix@R=25pt@C=50pt{  &  &  &\g 1 &  &  & \\
                              &  &  & 2 &  &  & \\
                            \g 1  &  &  &  &  &  & \g 1\\
                            2 &  &  & \h 1 &  &  & 2 \\
                             &  &  & \g 1 &  &  & \\
                              &  &  & 2 &  &  &
                            \ar _{\al }"3,1";"4,1"
                            \ar ^{\al }"1,4";"2,4"
                            \ar ^{\al }"5,4";"6,4"
                            \ar ^{\al }"3,7";"4,7"
                            \ar @/^/@{-->}^{X(a)}"3,1";"1,4"
                            \ar @/^/@{-->}^(0.7){X(a)}"4,1";"2,4"
                            \ar @/^/@{-->}^{X(c)}"3,1";"4,4"
                            \ar @{-->}^{X(c)}"4,1";"4,4"
                            \ar @/_/@{-->}_(0.6){X(e)}"3,1";"5,4"
                            \ar @/_/@{-->}_{X(e)}"4,1";"6,4"
                            \ar @/^/@{-->}^{X(b)}"1,4";"3,7"
                            \ar @/^/@{-->}_{X(b)}"2,4";"3,7"
                            \ar @{-->}^{X(d)}"4,4";"3,7"
                            \ar @/^/@{-->}_(0.4){X(f)}"5,4";"3,7"
                            \ar @/_/@{-->}_{X(f)}"6,4";"4,7"
     \save "2,4"+<0pt,-1cm>*\txt{$X(2)$}\restore
     \save "4,1"+<0pt,-1cm>*\txt{$X(1)$}\restore
     \save "4,4"+<0pt,1cm>*\txt{$X(3)$}\restore
     \save "4,7"+<0pt,-1cm>*\txt{$X(5)$}\restore
     \save "6,4"+<0pt,-1cm>*\txt{$X(4)$}\restore
                            }\]

Then by Theorem \ref{qv-pres-Gr} $\Gr (X)$ is presented by the quiver
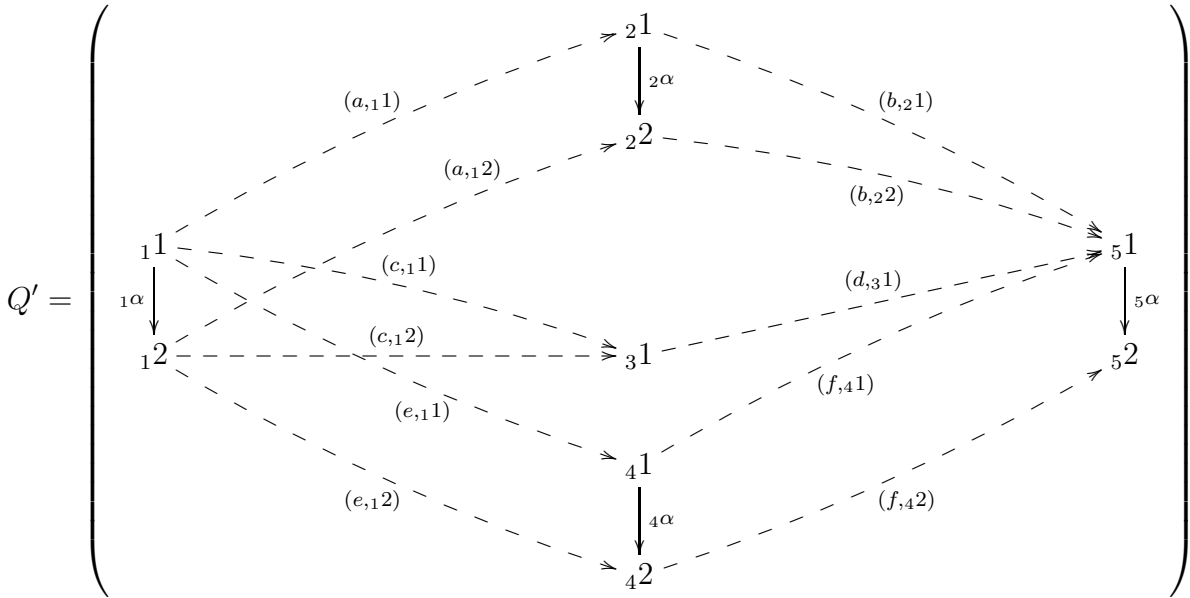
\[ Q'=\left(
\vcenter{\xymatrix@R=25pt@C=50pt{  &  &  & {}_21 &  &  & \\
                              &  &  & {}_22 &  &  & \\
                            {}_11  &  &  &  &  &  & {}_51\\
                            {}_12 &  &  & {}_31 &  &  & {}_52 \\
                             &  &  & {}_41 &  &  & \\
                              &  &  & {}_42 &  &  &
                            \ar _{{}_1\al }"3,1";"4,1"
                            \ar ^{{}_2\al }"1,4";"2,4"
                            \ar ^{{}_4\al }"5,4";"6,4"
                            \ar ^{{}_5\al }"3,7";"4,7"
                            \ar @/^/@{-->}^{(a,{}_11)}"3,1";"1,4"
                            \ar @/^/@{-->}^(0.7){(a,{}_12)}"4,1";"2,4"
                            \ar @/^/@{-->}^{(c,{}_11)}"3,1";"4,4"
                            \ar @{-->}^{(c,{}_12)}"4,1";"4,4"
                            \ar @/_/@{-->}_(0.6){(e,{}_11)}"3,1";"5,4"
                            \ar @/_/@{-->}_{(e,{}_12)}"4,1";"6,4"
                            \ar @/^/@{-->}^{(b,{}_21)}"1,4";"3,7"
                            \ar @/^/@{-->}_{(b,{}_22)}"2,4";"3,7"
                            \ar @{-->}^{(d,{}_31)}"4,4";"3,7"
                            \ar @/^/@{-->}_(0.4){(f,{}_41)}"5,4";"3,7"
                            \ar @/_/@{-->}_{(f,{}_42)}"6,4";"4,7"
                            }}\right)\]
with relations
\begin{eqnarray*}
R'&=&\{ \pi(ba,{}_11)-\pi(dc,{}_11),\pi(ba,{}_12)-\pi(dc,{}_12)\}\\
   & &\cup \{ (a,{}_iy){}_i\al -{}_j(a\al )(a,{}_ix)\mid a:i\to j\in Q_1,\al :x\to y\in Q_1^{(i)}\},
\end{eqnarray*}
where the new arrows are presented by broken arrows.

\end{exm}

\end{exm}

\begin{exm}[Semigroup case]
Define a category $I=\ang{Q\mid R}$ by setting
\[Q=\left(\xymatrix{ 1 \ar @(ur,dr)^{g}}\right),\quad R=\{(g^2,g^3)\}.\]
Then $I$ can be regarded as a
semigroup with the presentation
$\langle g\mid g^2=g^3\rangle$.
We define a functor $X:G\rar \kCat$ as follows.
Let $Q^{(1)}$ be the quiver
$$
\xymatrix{1 \ar[r]^{\al} & 2 \ar[r]^\be & 3}.
$$
and set \( X(1):= \k Q^{(1)}\),
and define an endofunctor $X(g)$ of $X(1)$ as 
the $\k$-functor induced by the vertex map $X(g)(1)=2, X(g)(2)=3, X(g)(3)=3$.
Then by Theorem \ref{qv-pres-Gr} $\gr (X)$ is presented by the quiver
\[ Q'=(\xymatrix{ 1& 2 & 3
                   \ar @<1mm>^\al "1,1";"1,2"
                   \ar @<-1mm>@{.>}_{(g,1)}"1,1";"1,2"
                   \ar @<1mm>^\be "1,2";"1,3"
                   \ar @<-1mm>@{.>}_{(g,2)}"1,2";"1,3"
                   \ar @(ur,dr)@{.>}^{(g,3)}
                   })
\]
with relations
\begin{eqnarray*}
R'&=&\{ (g,3)(g,2)(g,1)-(g,2)(g,1),(g,3)(g,3)(g,2)-(g,3)(g,2),\\
   & &(g,3)(g,3)(g,3)-(g,3)(g,3),(g,2)\al -\be (g,1),(g,3)\be -(g,2)\}.
\end{eqnarray*}
\end{exm}

\begin{exm}
Let $Q=(\xymatrix{ 1 \ar[r]^a & 2 })$ and $I:= \ang{Q}$.
Define functors $X, X' \colon I \to \kCat$ by the $\k$-linearizations of the following
quivers in frames and the $\k$-functors induced by the vertex maps expressed by dotted arrows
between them:
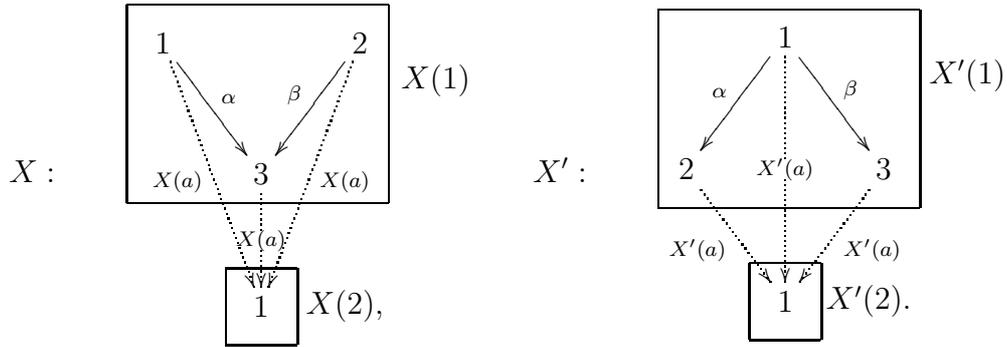
\def\gv{\save [].[drr]!C*++[F-]\frm{} \restore}
\def\h{\save [].[]!C*++[F]\frm{} \restore}
$$
X:\qquad \vcenter{\xymatrix@R=35pt{ \gv 1 &   & 2 \\
                                        & 3 &  \\
                                        & \h 1 & 
                         \ar ^{\al }"1,1";"2,2"
                         \ar _{\be }"1,3";"2,2"
                         \ar @{.>}_{X(a)}"1,1";"3,2"
                         \ar @{.>}^{X(a)}"1,3";"3,2"
                         \ar @{.>}|{X(a)}"2,2";"3,2"
\save "1,3"+<1cm,-0.5cm>*\txt{$X(1)$}\restore
\save "3,2"+<1.1cm,0cm>*\txt{$X(2)$,}\restore
}}
\qquad
X':\qquad \vcenter{\xymatrix@R=35pt{  \gv & 1 &   \\
                                    2 &   & 3 \\
                                      &\h  1 &  
                         \ar _{\al }"1,2";"2,1"
                         \ar ^{\be }"1,2";"2,3"
                         \ar @{.>}|{X'(a)}"1,2";"3,2"
                         \ar @{.>}_{X'(a)}"2,1";"3,2"
                         \ar @{.>}^{X'(a)}"2,3";"3,2"
\save "1,3"+<1.1cm,-0.5cm>*\txt{$X'(1)$}\restore
\save "3,2"+<1.1cm,0cm>*\txt{$X'(2)$.}\restore
}}
$$
Then by Theorem \ref{qv-pres-Gr} $\gr (X)$ is given by the following quiver with no relations
\[ \left( \vcenter{\xymatrix{ {}_11 &   & {}_12 \\
                                        & {}_13 &  \\
                                        & {}_21 & 
                         \ar ^{{}_1\al }"1,1";"2,2"
                         \ar _{{}_1\be }"1,3";"2,2"
                         \ar @{.>}_{(a,{}_11)}"1,1";"3,2"
                         \ar @{.>}^{(a,{}_12)}"1,3";"3,2"
                         \ar @{.>}_{(a,{}_13)}"2,2";"3,2"
                         }},(a,{}_13){}_1\al -(a,{}_11),(a,{}_13){}_1\be -(a,{}_12)\right) 
\iso \left(
\vcenter{\xymatrix{ {}_11 &   & {}_12 \\
                                        & {}_13 &  \\
                                        & {}_21 & 
                         \ar^{{}_1\al }"1,1";"2,2"
                         \ar _{{}_1\be }"1,3";"2,2"
                         \ar_{(a,{}_13)}"2,2";"3,2"
                         }}
\right),          
\]
and $\gr (X')$ is given by the following quiver with a commutativity relation
\[\left( \vcenter{\xymatrix{   & {}_11 &   \\
                                    {}_12 &   & {}_13 \\
                                      & {}_21 &  
                         \ar _{{}_1\al }"1,2";"2,1"
                         \ar ^{{}_1\be }"1,2";"2,3"
                         \ar @{.>}_{(a,{}_11)}"1,2";"3,2"
                         \ar @{.>}_{(a,{}_12)}"2,1";"3,2"
                         \ar @{.>}^{(a,{}_13)}"2,3";"3,2"
                         }},(a,{}_12){}_1\al -(a,{}_11),(a,{}_13){}_1\be -(a,{}_11)\right)
\iso \left(
\vcenter{\xymatrix{   & {}_11 &   \\
                                    {}_12 &  \circlearrowright & {}_13 \\
                                      & {}_21 &  
                         \ar _{{}_1\al }"1,2";"2,1"
                         \ar ^{{}_1\be }"1,2";"2,3"
                         \ar @{.>}_{(a,{}_12)}"2,1";"3,2"
                         \ar @{.>}^{(a,{}_13)}"2,3";"3,2"
                         }}
\right).
\]
By using the main theorem in \cite{Asa-b} derived equivalences between $X(1)$ and $X'(1)$
and between $X(2)$ and $X'(2)$ are glued together to have a derived equivalence
between $\Gr(X)$ and $\Gr(X')$.                 
\end{exm}

\end{document}